\documentclass[12pt]{article}
\usepackage{latexsym,amsmath,amsthm,amssymb,amscd,amsfonts}
\usepackage[usenames]{color}

\newtheorem{lemma}{Lemma}[section]
\newtheorem{proposition}[lemma]{Proposition}
\newtheorem{remark}[lemma]{Remark}
\newtheorem{theorem}[lemma]{Theorem}

\newtheorem{corollary}[lemma]{Corollary}

\newtheorem*{remark*}{Remark}

\def\Symp{\text{Symp}}
\def\Ham{\text{Ham}}
\def\Diff{\text{Diff}}
\def\id{\text{id}}
\def\Flux{\text{Flux}}

\def\Map{\text{Map}}

\makeatletter \@addtoreset {equation}{section}

\renewcommand\theequation
  {\ifnum \c@subsection>\z@ \arabic{section}.\arabic{subsection}.\arabic{equation}
  \else \arabic{section}.\arabic{equation} \fi}
\makeatother

\begin{document}

\title{Towards the $C^0$ flux conjecture}

\author{Lev Buhovsky$^{1}$}

\footnotetext[1]{The author also uses the spelling ``Buhovski"
for his family name.}

\date{\today}
\maketitle

\begin{abstract}
In this note, we generalise a result of Lalonde, McDuff and Polterovich concerning the $ C^0 $ flux conjecture, thus confirming the conjecture in new cases of a symplectic manifold. Also, we prove the continuity of the flux homomorphism on the space of smooth symplectic isotopies endowed with the $ C^0 $ topology, which implies the $ C^0 $ rigidity of Hamiltonian paths, conjectured by Seyfaddini.
\end{abstract}

\noindent

\section{Introduction and main results}

The celebrated Eliashberg-Gromov rigidity theorem~\cite{E1,E2,G} states that on any closed symplectic manifold $ (M,\omega) $, the group $ \Symp(M,\omega) $ of symplectomorphisms of $ M $ is $C^0$-closed inside the group $ \Diff(M) $ of diffeomorphisms of $ M $. A related natural conjecture (called the $ C^0 $ flux conjecture) was raised in Banyaga's foundational paper~\cite{B}: is the group $ \Ham(M,\omega) $ of Hamiltonian diffeomorphisms of $ M $ $ C^0 $-closed inside $ Symp_0(M,\omega) $, the connected component of identity in $ \Symp(M,\omega) $?

The reader may wonder why it is asked if $ \Ham(M,\omega) $ is $ C^0 $-closed in $ \Symp_0(M,\omega) $ rather than $ \Symp(M,\omega) $. The difficulty in addressing the latter question is that, although the Eliashberg-Gromov rigidity theorem tells us that $ \Symp(M,\omega) $ is $ C^0 $-closed in the group $ \Diff(M) $ of diffeomorphisms of $ M $, it is not known if $ \Symp_0(M,\omega) $ is $ C^0 $-closed in $ \Symp(M,\omega) $. To avoid this difficulty the $ C^0 $ flux conjecture is usually stated for $ \Symp_0(M,\omega) $.

A weak form of the $ C^0 $ flux conjecture is the $ C^1 $ flux conjecture, which states that $ \Ham(M,\omega) $ is $ C^1 $-closed in $ \Symp_0(M,\omega) $. This statement is equivalent to that the flux group $ \Gamma \subset H^1(M,\mathbb{R}) $ is discrete. Some cases of the $ C^1 $ flux conjecture were proven in~\cite{B}, \cite{L-M-P}, \cite{M}; it was finally confirmed in full generality by Ono~\cite{O-1}. However, the $ C^0 $ flux conjecture still remains open in case of a general symplectic manifold. It has been confirmed by Lalonde, McDuff and Polterovich in certain cases~\cite{L-M-P} (these cases are described below). Also, Humili\`ere and Vichery established more cases of the $ C^0 $ flux conjecture in their joint work~\cite{Hu-V}. 

A different weak form of the $ C^0 $ flux conjecture (the ``$ C^0 $ rigidity of Hamiltonian paths") was proposed by Seyfaddini~\cite{S}: Is it true that on any closed symplectic manifold, the space of smooth Hamiltonian isotopies is $ C^0 $ close in the space of smooth symplectic isotopies? In~\cite{S}, Seyfaddini showed that a symplectic isotopy which is a $ C^0 $ limit of a sequence of Hamiltonian isotopies is itself Hamiltonian, provided that the corresponding sequence of generating Hamiltonians is a Cauchy sequence in the $ L^{(1,\infty)} $ topology. 

The results of this note are concerned with the $ C^0 $ flux conjecture, and with the mentioned conjecture of Seyfaddini (the $ C^0 $ rigidity of Hamiltonian paths).

\subsection{The $ C^0 $ flux conjecture}

We denote $ H = \Ham(M,\omega) \subseteq G = \Symp_0(M,\omega) $, by $ \widetilde{G} = \widetilde{\Symp}_0 (M,\omega) $ we denote the universal cover of $ G = \Symp_0(M,\omega) $, and by $ \widetilde{H} \subseteq \widetilde{G} $ we denote those elements of $ \widetilde{G} $ whose endpoint belongs to $ H $. Next, by $ H_0 \subseteq G $ we denote the $ C^0 $ closure of $ H $ inside $ G $, and by $ \widetilde{H}_0 \subseteq \widetilde{G} $ we denote the lift of $ H_0 $ to $ \widetilde{G} $. Also, we use the notation $ \Map_0(M) $ for the connected component of the identity in the space of all smooth maps $ M \rightarrow M $.

Denote by $ \Gamma \subset H^1(M,\mathbb{R}) $ the flux group, i.e. the image of $ \widetilde{H} $ (or, equivalently, of $ \pi_1(G) $) under the flux homomorphism, and by $ \Gamma_0 \subset H^1(M,\mathbb{R}) $ the image of $ \widetilde{H}_0 $ under the flux homomorphism. It is not hard to see that the $ C^0 $ flux conjecture is equivalent to the equality $ \Gamma_0 = \Gamma $ (this follows from the well-known fact that for a path $ \phi^t $, $ t \in [0,1] $ of symplectic diffeomorphisms, its endpoint $ \phi^1 $ belongs to $ H $ if and only if its flux belongs to $ \Gamma $). The restriction of the flux homomorphism to $ \pi_1(G) $ admits a natural extension to a homomorphism (which we again call flux homomorphism) from $ \pi_1(\Map_0(M)) $ to the $ H^1(M,\mathbb{R}) $. 
Following~\cite{L-M-P}, we denote by $ \Gamma_{top} $ the image of $ \pi_1(\Map_0(M)) $ under the flux homomorphism. Consider the evaluation homomorphism $ ev : \pi_1(\Map_0(M)) \rightarrow \pi_1(M) $. For any $ a \in \pi_1(M) $ we denote by $ \Gamma^a_{top} \subseteq \Gamma $ the image of $ ev^{-1}(a) \subseteq \pi_1(\Map_0(M)) $ under the flux homomorphism. 

The following result was proved in~\cite{L-M-P}:

\begin{theorem} \label{T:L-M-P}
If $ M $ is Lefschetz, then $ \Gamma_0 \subseteq \Gamma_{top} $. 
\end{theorem}

As a consequence, Lalonde, McDuff and Polterovich conclude:
 
\begin{corollary} \label{C:L-M-P}
Assume that $ M $ is Lefschetz and that $ \Gamma_{top} = \Gamma $. Then the $ C^0 $ flux conjecture holds for $ M $. 
\end{corollary}  

As an example, one can take $ M $ to be a closed K\"ahler manifold of nonpositive curvature such that its fundamental group has a trivial center. As another example, one can take the $2n$-dimensional torus with a translation invariant symplectic structure. See~\cite{L-M-P} for more details.

Now we turn to our results. Our main result is:

\begin{theorem} \label{T:Main-result}
Let $ (M,\omega) $ be a closed symplectic manifold. Then $ \Gamma_0 \subseteq \Gamma_{top} + \overline{\Gamma^e_{top}} $, where $ e \in \pi_1(M) $ is the identity, and $ \overline{\Gamma^e_{top}} \subseteq H^1(M,\mathbb{R}) $ is the closure of $ \Gamma^e_{top} $ inside $ H^1(M,\mathbb{R}) $.
\end{theorem} 

\noindent As a result, we obtain the following corollary:

\begin{corollary} \label{C:main}
Let $ (M,\omega) $ be a closed symplectic manifold such that $ \Gamma_{top} = \Gamma $. Then the $ C^0 $ flux conjecture holds for $ M $.
\end{corollary}

Indeed, if $ \Gamma_{top} = \Gamma $, then since $ \Gamma $ is closed (the closeness of $ \Gamma $ is exactly the statement of the $ C^1 $ flux conjecture proven in~\cite{O-1}), it follows that $ \Gamma_{top} + \overline{\Gamma^e_{top}} = \Gamma $ and hence $ \Gamma_0 = \Gamma $ by Theorem~\ref{T:Main-result}. 

In particular, the $ C^0 $ flux conjecture holds for a closed symplectically aspherical symplectic manifold $ (M,\omega) $ such that the fundamental group $ \pi_1(M) $ has a trivial center. Indeed, in this case, since the center of $ \pi_1(M) $ is trivial, we get $ ev( \pi_1(\Map_0(M))) = \{ e \} $ and so $ \Gamma_{top} = \Gamma^e_{top} $, and moreover, since $ M $ is symplectically aspherical, we conclude that $ \Gamma^e_{top} = \{ 0 \} $. Therefore $ \Gamma_{top} = 0 $ and hence $ \Gamma_{top} = \Gamma = \{ 0 \} $.

As another example, we get that the $ C^0 $ flux conjecture holds for any product $ (M,\omega) = (\mathbb{T}^{2k} \times N, \sigma \oplus \tau) $, where $ (\mathbb{T}^{2k},\sigma) $ is a symplectic torus with a translation invariant $ \sigma $, and $ (N,\tau) $ is a closed symplectically atoroidal symplectic manifold. Indeed, since $ \mathbb{T}^{2k} $ is symplectically aspherical, and $ N $ is symplectically atoroidal, it follows that for any $ a = (b,c)  \in \pi_1(M) \cong \pi_1(\mathbb{T}^{2k}) \times \pi_1(N) $ and any $ \mathfrak{a} \in \pi_1(\Map_0(M)) $ with $ ev(\mathfrak{a}) = a $, the flux of $ \mathfrak{a} $ is uniquely determined by $ b $. Moreover, since translations of the torus $ \mathbb{T}^{2k} $ generate a large enough subgroup of symplectomorphisms of $ (\mathbb{T}^{2k},\sigma) $, for any $ b \in \pi_1(\mathbb{T}^{2k}) $ we can find an element $ \mathfrak{b} \in \pi_1(G) = \pi_1(\Symp_0(M,\omega)) $ such that $ ev(\mathfrak{b}) = (b,0) $. Therefore we conclude $ \Gamma_{top} = \Gamma $.

\begin{remark}
The reader may ask if there exist examples of closed symplectic manifolds for which $ \Gamma_{top} \neq \Gamma $. The following construction is due to Seidel~\cite{Sei-1,Sei-2}. Let $ (N,\omega_N) $ be a closed symplectic manifold with $ H^1(N,\mathbb{R}) = 0 $, let $ \psi : N \rightarrow N $ be a symplectic diffeomorphism which is smoothly isotopic to the identity, but which is not isotopic to the identity via a smooth path of symplectic diffeomorphisms. Look at the symplectic mapping torus $ E = E_{\psi} $ of $ \psi $, which is the total space of the fibration over the two-torus with fibre $ N $ and monodromy $ \psi $ in one direction, or explicitly, $$ E = \mathbb{R}^2 \times N / (p,q,x) \sim (p-1,q,x) \sim (p,q-1,\psi(x)), $$ $$ \omega_{E} = dp \wedge dq + \omega_N .$$ Because $ \psi $ is smoothly isotopic to the identity, the fibration $ E \rightarrow \mathbb{T}^2 $ is trivial as a smooth one, and it is easy to see that for $ E $ we have $ \Gamma_{top} = H^1(E,\mathbb{Z}) $. However, it is possible that for $ E $ we have $ \Gamma \neq H^1(E,\mathbb{Z}) $. The closed 1-form $ dp $ on $ E $ generates the symplectic vector field $ \frac{\partial}{\partial q} $, whose time-1 map is  $ \phi(p,q,x) = (p,q+1,x) = (p,q,\psi(x)) $. If $ \phi $ turns out to be a non-Hamiltonian diffeomorphism, then we get that $ [dp] \notin \Gamma $, so in particular $ \Gamma \neq \Gamma_{top} $. One way of detecting this is by looking at the Floer cohomology $ HF^*(\psi) $. That is, if we are in a situation when $ HF^*(\psi) $ has total rank different from that of $ H^*(N) $, then $ HF^*(\phi) \cong H^*(\mathbb{T}^2) \otimes HF^*(\psi) $ is not isomorphic to $ H^*(E) $, and hence in particular $ \phi $ is a non-Hamiltonian diffeomorphism (here we consider cohomologies with coefficients in the corresponding Novikov ring). 
\end{remark}

\subsection{The $ C^0 $ rigidity of Hamiltonian paths}

Our next result is: 

\begin{theorem} \label{T:flux-lipschitz}
Let $ (M,\omega) $ be a closed symplectic manifold. Fix a Riemannian metric $ g $ on $ M $, which induces a distance function $ d : M \times M \rightarrow \mathbb{R} $, which in turn, induces a distance $ d $ between maps $ M \rightarrow M $: for any $ f,h : M \rightarrow M $ we set $ d(f,h) = \sup_{x \in M} d(f(x),h(x)) $. Fix a norm $ | \cdot | $ on $ H^1(M,\mathbb{R}) $. Then there exist constants $ c = c(M,\omega,g), C = C(M,\omega,g, | \cdot |) $, such that for any path $ \phi^t $, $ t \in [0,1] $ of symplectomorphisms of $ M $, $ \phi^0 = \id_M $, $ \phi^1 = \phi $, with $ \max_{t \in [0,1]} d(\id_M,\phi^t) < c $, we have $ | \Flux(\{ \phi^t \}) | \leqslant C d(\id_M,\phi) $. 
\end{theorem}

Theorem~\ref{T:flux-lipschitz} has a direct corollary:

\begin{corollary} \label{C:rigidity-ham-paths}
{\bf 1)} On any closed symplectic manifold, the flux homomorphism is continuous with respect to the $ C^0 $ distance between smooth paths of symplectomorphisms. \\
{\bf 2)} $ C^0 $ rigidity of Hamiltonian paths: on any closed symplectic manifold, the space of smooth Hamiltonian isotopies of $ M $ is $ C^0 $-closed in the space of smooth symplectic isotopies of $ M $. This confirms the mentioned above conjecture of Seyfaddini.
\end{corollary}

Let us remark, that there is another weak version of the $ C^0 $ flux conjecture, which is due to Seyfaddini, and it concerns with the topological (continuous) Hamiltonian dynamics initially introduced by Oh and M\"uller~\cite{Oh-M}: Is it true that any Hamiltonian homeomorphism (in the sense of~\cite{Oh-M}) which belongs to $ \Symp_0(M,\omega) $, is in fact a Hamiltonian diffeomorphism?

\subsection*{Notations}
Let $ A > 0 $. We denote by $ B(A) \subset \mathbb{R}^2 $ the open euclidean disc centered at the origin having area $ A $, i.e. $ B(A) = \{ z \in \mathbb{R}^2 \, | \, \pi |z|^2 < A \} $. We denote $ S(A) = \partial B(A) = \{ z \in \mathbb{R}^2 \, | \, \pi |z|^2 = A \} $, the euclidean circle centered at the origin enclosing a disc of area $ A $. Also, we use the notation $ B'(A) = B(A) \setminus \{ 0 \} \subset \mathbb{R}^2 $ for the punctured disc. On $ T^{*}S^1 $ with canonical coordinates $ (q,p) $, where $ q \in \mathbb{R} / \mathbb{Z} $, $ p \in \mathbb{R} $, and with the standard symplectic form $ dp \wedge dq $, we will use the notation $ S^1 \subset T^* S^1 $ for the zero-section, and we denote $ W(A) = \{ (q,p) \, | \, |p| < A \} \subset T^* S^1 $, so that $ W(A) $ is a neighbourhood of the zero-section in $ T^* S^1 $ having area $ 2A $. 

\subsection*{Acknowledgements}
I thank  Matthew Strom Borman, Yakov Eliashberg, Jonny Evans, Vincent Humuliere, Jarek Kedra, Dusa McDuff, Emmanuel Opshtein, Timothy Perutz, Leonid Polterovich, Paul Seidel, Sobhan Seyfaddini and Mikhail Sodin for useful discussions, comments and suggestions. I thank the organisers of the workshop ``J-holomorphic Curves in Symplectic Geometry, Topology, and Dynamics" which was held on April 29 - May 10, 2013, in Montreal, Canada, for providing an enjoyable and lively research atmosphere.

\section{Proofs}

Consider the evaluation homomorphism $ ev : \pi_1(\Map_0(M)) \rightarrow \pi_1(M) $. In the next lemma we show that its restriction to $ \pi_1(G) $ can be naturally extended to a homomorphism $ ev' : \widetilde{H}_0 \rightarrow \pi_1(M) $:

\begin{lemma} \label{L:homomo}
The homomorphism $ ev |_{\pi_1(G)} : \pi_1(G) \rightarrow \pi_1(M) $ admits a natural extension to a homomorphism $ ev' : \widetilde{H}_0 \rightarrow \pi_1(M) $.
\end{lemma}
\begin{proof}
Let us first present a construction of $ ev' $. Fix a Riemannian metric $ g $ on $ M $. Let $ \tilde{\phi} \in \widetilde{H} $, let $ \phi^t , t \in [0,1] $ be a path of symplectomorphisms representing $ \tilde{\phi} $, and denote $ \phi = \phi^1 $. Consider a Hamiltonian diffeomorphism $ \psi \in H $ such that $ \psi $ is sufficiently $ C^0 $-close to $ \phi $ (it is possible to find such a Hamiltonian diffeomorphism, since by our assumption $ \phi $ lies in the $ C^0 $ closure of $ H $ inside $ G $). Let $ \psi^t , t \in [0,1] $ be a Hamiltonian isotopy of $ M $, such that $ \psi^1 = \psi $. Define a continuous loop $ f^t $, $ t \in [0,3] $ in $ \Map_0(M) $, such that $ f^t = \phi^{t} $ for $ t \in [0,1] $, such that for any $ x \in M $, the path $ f^t(x) $, $ t \in [1,2] $ is the shortest $ g $-geodesic connecting $ \phi(x) $ and $ \psi (x) $, and such that $ f^t = \psi^{3-t} $ for $ t \in [2,3] $. We now define $ ev'(\phi) $ to be the value of the evaluation map $ ev $ at the loop $ f^t $, $ t \in [0,3] $.

Let us show that the definition does not depend on the choice of $ \psi $ and of the path $ \psi^t $, $ t \in [0,1] $. Let $ \psi^t_1 $, $ \psi^t_2 $, $ t \in [0,1] $ be two Hamiltonian isotopies of $ M $, such that $ \psi_1 = \psi_1^1 $ and $ \psi_2 = \psi^1_2 $ are sufficiently $ C^0 $-close to $ \phi $. Define the corresponding loops $ f_1^t $, $ f_2^t $, $ t \in [0,3] $ as above. Define the loop $ h_1^t $, $ t \in [0,6] $ in $ \Map_0(M) $ by $ h_1^t = f_1^{3-t} $, $ t \in [0,3] $ and $ h_1^t = f_2^{t-3} $, $ t \in [3,6] $. It is enough to show that the value of $ ev $ at the loop $ h_1^t $, $ t \in [0,6] $ equals $ e \in \pi_1(M) $. Clearly, the loop $ h_1^t $, $ t \in [0,6] $ is homotopic to the loop $ h_2^t $, $ t \in [0,4] $, where $ h_2^t = \psi_1^t $ for $ t \in [0,1] $, where for any $ x \in M $ the path $ h_2^t(x) $, $ t \in [1,2] $ is the shortest $ g $-geodesic between $ \psi_1(x) $ and $ \phi(x) $, and the path $ h_2^t(x) $, $ t \in [2,3] $ is the shortest $ g $-geodesic between $ \phi(x) $ and $ \psi_2(x) $, and finally $ h_2^t = \psi_2^{4-t} $ for $ t \in [3,4] $. Also, since $ \psi_1 $ and $ \psi_2 $ are $ C^0 $ close to $ \phi $, it follows that the loop $ h_2^t $, $ t \in [0,4] $ is homotopic to the loop $ h_3^t $, $ t \in [0,3] $, where $ h_3^t = \psi_1^t $ for $ t \in [0,1] $, where for any $ x \in M $ the path $ h_3^t(x) $, $ t \in [1,2] $ is the shortest $ g $-geodesic between $ \psi_1(x) $ and $ \psi_2(x) $, and where $ h_3^t = \psi_2^{3-t} $ for $ t \in [2,3] $. It is enough to show that the value of the evaluation map $ ev $ at $ h_3^t $, $ t \in [0,3] $ equals $ e \in \pi_1(M) $. Now pick some  increasing bijective smooth function $ \nu : [0,1] \rightarrow [0,1] $, such that the derivatives of $ \nu $ of all orders vanish at $ 1 \in [0,1] $, and look at the smooth Hamiltonian flow $ h^t $, $ t \in [0,2] $ defined by $ h^{t} = \psi_1^{\nu(t)} $ for $ t \in [0,1] $ and $ h^{t} = \psi_2^{\nu(2-t)} $ for $ t \in [1,2] $. Then by the solution of the Arnold's conjecture~\cite{C-Z,F1,F2,O-2,Ho-S,Li-T,F-O-1,R,F-O-2}\footnote{Strictly speaking, except for the case of a semi-positive symplectic manifold, the existing proofs of the Arnold's conjecture are conditional to the virtual cycle techniques which are not yet accepted by all the experts.}, 
the time-$2$ map $ h^2 $ of the Hamiltonian flow $ h^t $, $ t \in [0,2] $ has a fixed point $ p \in M $, such that its trajectory under the flow is a contractible loop. Hence we get $ \psi_1(p) = \psi_2(p) $, and as a result, the loop $ t \mapsto h_3^t(p) $, $ t \in [0,3] $ is contractible. Therefore the value of the evaluation map $ ev $ at the loop $ h_3^t $, $ t \in [0,3] $ equals $ e \in \pi_1(M) $.

Finally, it is easy to see the independence of $ ev' $ of the choice of metric $ g $, and that $ ev' $ is a homomorphism.

\end{proof}

The main step in the proof of Theorem~\ref{T:Main-result} is the following proposition:

\begin{proposition} \label{P:main-step-proof-main-result}
Let $ (M,\omega) $ be a closed symplectic manifold, and let $ \tilde{\phi} \in \widetilde{H}_0 $. Then $ \Flux (\tilde{\phi}) \in 
\overline{\Gamma^{a}_{top}} $, where $ a = ev'(\tilde{\phi}) \in \pi_1(M) $.
\end{proposition}
\begin{proof} 
Choose a Riemannian metric $ g $ on $ M $. Denote by $ \phi \in G $ the endpoint of $ \tilde{\phi} $, and let $ \phi^t $, $ t \in [0,1] $ be a symplectic isotopy of $ M $ representing $ \tilde{\phi} $, such that $ \phi^0 = \id_M $, $ \phi^1 = \phi $. Choose a smooth closed loop $ \gamma : [0,1]  \rightarrow M $, and define the loop $ \alpha = \phi \circ \gamma $. There exists a neighbourhood $ U $ of $ \alpha ([0,1]) $ which is symplectomorphic to the product of a neighbourhood of the zero-section in $ T^* S^1 $ and $ n-1 $ small $2$-dimensional discs, i.e. there exists $ \epsilon > 0 $, such that for $ W(\epsilon) = \{ (q,p) \in T^* S^1 \, | \, |p| < \epsilon \} \subset T^* S^1 $ (here $ S^1 \cong \mathbb{R} / \mathbb{Z} $, so that the symplectic area of $ W(\epsilon) $ in $ T^* S^1 $ is $ 2 \epsilon $) and for the standard $2$-dimensional disc $ B(\epsilon) \subset \mathbb{R}^2 $ of area $ \epsilon $ centered at the origin, we have a symplectic embedding $ \iota : W(\epsilon) \times B(\epsilon)^{\times n-1} \rightarrow M $, such that $ S^1 \times \{ 0 \} \times ... \times \{ 0 \} $ is mapped onto $ \alpha([0,1]) $, where $ S^1 \subset T^{*} S^1 $ is the zero-section. We set $ U = \iota(W(\epsilon) \times B(\epsilon)^{\times n-1}) $. 

Now let $ \psi \in H $ be sufficiently $ C^0 $-close to $ \phi $, and let $ \psi^t $, $ t \in [0,1] $ be a Hamiltonian isotopy on $ M $ such that $ \psi^1 = \psi $. Define, as in the proof of Lemma~\ref{L:homomo}, 
a continuous loop $ f^t $, $ t \in [0,3] $ in $ \Map_0(M) $, such that $ f^t = \phi^{t} $ for $ t \in [0,1] $, such that for any $ x \in M $, the path $ f^t(x) $, $ t \in [1,2] $ is the shortest $ g $-geodesic connecting $ \phi(x) $ and $ \psi (x) $, and such that $ f^t = \psi^{3-t} $ for $ t \in [2,3] $. Then since $ \psi $ is sufficiently $ C^0 $-close to $ \phi $, Lemma~\ref{L:homomo} tells us that the value of $ ev $ at the loop $ f^t $, $ t \in [0,3] $ equals to $ ev'(\tilde{\phi}) $. Define smooth cylinders $ w,u,v : [0,1] \times [0,1] \rightarrow M $ by $ w(s,t) = f^t(\gamma(s)) = \phi^t(\gamma(s)) $, $ u(s,t) = f^{t+1}(\gamma(s)) $, $ v(s,t) = f^{t+2}(\gamma(s)) = \psi^{1-t}(\gamma(s)) $, for $ (s,t) \in [0,1] \times [0,1] $. The loop $ \beta := \psi \circ \gamma $ is $ C^0 $-close to the loop $ \alpha = \phi \circ \gamma $, hence the image $ \beta([0,1]) $ lies inside $ U $, and moreover the image $ u([0,1] \times [0,1]) $ lies inside $ U $.  The union of the images of $ w,  u , v $ is a torus, which is the isotopy of the loop $ \gamma $ via the path $ f^t $, $ t \in [0,3] $. We therefore have the equality $ \omega(w) + \omega(u) + \omega(v) = \Flux(f^t)(\gamma) $. We have $ \omega(w) = \Flux(\phi^t)(\gamma) $, and $ \omega(v) = 0 $ since the isotopy $ \psi^{t} $, $ t \in [0,1] $ is Hamiltonian. Thus we get $ \Flux(f^t)(\gamma) - \Flux(\phi^t)(\gamma) = \omega(u) $. Hence it is enough to show that for any initially chosen loop $ \gamma : [0,1] \rightarrow M $ as above, the symplectic area $ \omega(u) $ is arbitrarily small, provided that $ \psi $ is sufficiently $ C^0 $-close to $ \phi $.

Let us show this by a contradiction. Assume the contrary, i.e. that there exists some $ \epsilon' > 0 $, such that one can find $ \psi $ arbitrarily $ C^0 $-close to $ \phi $ for which we have $ | \omega(u) | \geqslant \epsilon' $. WLOG we may assume that $ \epsilon' < \epsilon $. 
Now pick $ \psi \in H $ which is sufficiently $ C^0 $-close to $ \phi $ and such that $ | \omega(u) | \geqslant \epsilon' $.
Recall that we have a symplectic embedding $ \iota : W(\epsilon) \times B(\epsilon)^{\times n-1} \rightarrow M $, such that $ S^1 \times \{ 0 \} \times ... \times \{ 0 \} $ is mapped onto $ \alpha([0,1]) $, and we have $ U = \iota(W(\epsilon) \times B(\epsilon)^{\times n-1}) $. Put $ \delta = \epsilon' /2 $, and consider the Lagrangian $ L = S^{1} \times S(\delta) \times ... \times S(\delta) \subset W(\epsilon') \times B(\epsilon')^{\times n-1} \subset W(\epsilon) \times B(\epsilon)^{\times n-1} $, where $ S(\delta) = \{ z \in \mathbb{R}^2 \, | \, \pi |z|^2 = \delta \} $ is the circle on $ \mathbb{R}^2 $ centered at the origin enclosing a disc of area $ \delta $. Since $ \psi $ is sufficiently $ C^0 $-close to $ \phi $, it follows that $ \psi \circ \phi^{-1} $ is sufficiently $ C^0 $-close to $ \id_M $, and in particular $ \psi \circ \phi^{-1} ( \iota (W(\epsilon') \times B(\epsilon')^{\times n-1}) ) \subset U = \iota (W(\epsilon) \times B(\epsilon)^{\times n-1}) $, and $ \psi \circ \phi^{-1} (\iota(L)) \subset \iota (W(\epsilon'/4) \times B'(\epsilon)^{\times n-1}) $ (recall that $ B'(\epsilon) = B(\epsilon) \setminus \{ 0 \} \subset \mathbb{R}^2 $ is the open punctured euclidean disc centered at the origin having area $ \epsilon $). Now, if we denote $ \tilde{L} := \iota^{-1}(\psi \circ \phi^{-1} (\iota(L))) \subset W(\epsilon'/4) \times B'(\epsilon)^{\times n-1} $, then $ \tilde{L} $ is a Lagrangian which is $ C^0 $ close to $ L $, and $ \pi_1(\tilde{L}) $ is generated by the loops $ \tilde{\beta}_1, ..., \tilde{\beta}_n $, which are the push-forwards of the loops $ \beta_1,...,\beta_n $ on $ L = S^1 \times S(\delta) \times ... \times S(\delta) $, such that the homotopy classes of $ \beta_1,...,\beta_n $ in $ \pi_1(L) $ correspond to the factors $ S^1,S(\delta),...,S(\delta) $. Consider the $1$-form $ \lambda = p_0 dq_0 + \frac{1}{2}(x_1 dy_1 - y_1 dx_1) + ... + \frac{1}{2}(x_{n-1} dy_{n-1} - y_{n-1} dx_{n-1}) =  p_0 dq_0 + \frac{1}{2} r_1^2 d\theta_1 + ... + \frac{1}{2} r_{n-1}^2 d\theta_{n-1} $ on $ W(\epsilon) \times B(\epsilon)^{\times n-1} $. Then $ d \lambda = dp_0 \wedge dq_0 + dx_1 \wedge dy_1 + ... + dx_{n-1} \wedge dy_{n-1} $ is the standard symplectic form on $ W(\epsilon) \times B(\epsilon)^{\times n-1} $.
Since the map $ \iota^{-1} \circ \psi \circ \phi^{-1} \circ \iota $ is well defined on $ W(\epsilon') \times B(\epsilon')^{\times n-1} $, and is symplectic, the $1$-form $ (\iota^{-1} \circ \psi \circ \phi^{-1} \circ \iota)^* \lambda - \lambda $ on $ W(\epsilon') \times B(\epsilon')^{\times n-1} $ is closed, and its evaluation at the loop $ S^1 \times \{ 0 \} \times ... \times \{ 0 \} \subset W(\epsilon') \times B(\epsilon')^{\times n-1} $ equals to the symplectic area $ \omega(u) $. Hence at the level of cohomology we have $ [ (\iota^{-1} \circ \psi \circ \phi^{-1} \circ \iota)^* \lambda - \lambda ] = \omega(u) [dq] $. In particular, we have $ \lambda(\tilde{\beta}_1) =  (\iota^{-1} \circ \psi \circ \phi^{-1} \circ \iota)^* \lambda (\beta_1) = \omega(u) + \lambda(\beta_1) = \omega(u)  $, and for $ 2 \leqslant j \leqslant n $ we have $ \lambda(\tilde{\beta}_j) =  (\iota^{-1} \circ \psi \circ \phi^{-1} \circ \iota)^* \lambda (\beta_j) = \lambda(\beta_j) = \delta $. 
Now let us present two possible ways of finishing the proof via arriving to a contradiction. The second way is easier and it was suggested by Seyfaddini.

{\bf First way:} 
Consider the $1$-form $ \lambda' =  \frac{1}{2}(x_0 dy_0 - y_0 dx_0) + \frac{1}{2}(x_1 dy_1 - y_1 dx_1) + ... + \frac{1}{2}(x_{n-1} dy_{n-1} - y_{n-1} dx_{n-1}) = \frac{1}{2} r_0^2 d \theta_0 + \frac{1}{2} r_1^2 d\theta_1 + ... + \frac{1}{2} r_{n-1}^2 d\theta_{n-1} $ on $ B(\epsilon' / 2) \times B'(\epsilon)^{\times n-1} $ endowed with coordinates $ (x_0,y_0,...,x_{n-1},y_{n-1}) $, where $ x_i = r_i \cos \theta_i $, $ y_i = r_i \sin \theta_i $ for $ i = 0,1,...,n-1 $. We have that $ d \lambda' = dx_0 \wedge dy_0 + dx_1 \wedge dy_1 + ... + dx_{n-1} \wedge dy_{n-1} $ is the standard symplectic form on $ B(\epsilon' / 2) \times B'(\epsilon)^{\times n-1} $. Consider the embedding $ \iota' : W(\epsilon' / 4) \times B'(\epsilon)^{\times n-1} \hookrightarrow B(\epsilon' / 2) \times B'(\epsilon)^{\times n-1} $, given by $ \iota'(q_0,p_0,x_1,y_1,...,x_{n-1},y_{n-1}) = (r_0,\theta_0,x_1,y_1,...,x_{n-1},y_{n-1}) $, $ \pi r_0^2 = p_0 + \epsilon'/4 $, $ \theta_0 = 2 \pi q_0 $. Then we have $ (\iota')^* \lambda' = \lambda + \frac{1}{4}\epsilon' dq_0 $. Hence for the Lagrangian $ \hat{L} = \iota'(\tilde{L}) \subset B(\epsilon' / 2) \times B'(\epsilon)^{\times n-1} \subset \mathbb{R}^{2} \times B'(\epsilon)^{\times n-1} $ and the loops $ \hat{\beta_j} := \iota' \circ \tilde{\beta}_j $, $ j = 1,...,n $, generating $ \pi_1(\hat{L}) $, we have $ \lambda'(\hat{\beta}_1) = \omega(u) + \epsilon'/4 $. Therefore, if we consider $ \hat{L} $ as a Lagrangian inside $ \mathbb{R}^{2} \times B'(\epsilon)^{\times n-1} $ endowed with the standard symplectic form, then it follows that the symplectic area of any disc in $ \mathbb{R}^{2} \times B'(\epsilon)^{\times n-1} $ with boundary on $ \hat{L} $, is an integer multiple of $ \omega(u) + \epsilon'/4 $, and hence its absolute value is $ \geqslant | \omega(u) | - \epsilon' / 4 \geqslant \epsilon' - \epsilon' /4 = 3 \epsilon' / 4 > \epsilon' /2 $. By the Chekanov's theorem~\cite{Ch}, the displacement energy $ e(\hat{L}) $ of $ \hat{L} $ inside $ \mathbb{R}^{2} \times B'(\epsilon)^{\times n-1} $ is greater than or equal to the minimal area of a non-constant holomorphic disc with boundary on $ \hat{L} $. Hence we conclude that $ e(\hat{L}) > \epsilon'/2 $. But on the other hand, since $ \hat{L} \subset B(\epsilon'/2) \times B'(\epsilon)^{\times n-1} $, one can clearly displace $ \hat{L} $ with a Hamiltonian isotopy of energy $ \epsilon'/2 $. Contradiction.

{\bf Second way:} 
Consider the Liouville form $ \lambda' = p_0 dq_0 + p_1 dq_1 + ... + p_{n-1} dq_{n-1} $ on $ T^* \mathbb{T}^{n} $. Let $ \iota' : W(\epsilon' / 4) \times B'(\epsilon)^{\times n-1} \rightarrow T^* \mathbb{T}^n $ be the symplectic embedding given by $$ \iota'(q_0,p_0,r_1,\theta_1,...,r_{n-1},\theta_{n-1}) = (q_0,p_0,q_1,p_1,...,q_{n-1},p_{n-1}) ,$$ $ \pi r_i^2 - \delta = p_i $, $ \theta_i = 2 \pi q_i $ for $ i = 1,2,...,n-1 $. The image of $ \iota' $ lies inside $ W(\epsilon'/4) \times (T^* S^1)^{\times n-1} = W(\epsilon'/4) \times T^* \mathbb{T}^{n-1} $. We have $ (\iota')^* \lambda' = \lambda - \frac{\delta}{2\pi} d\theta_1 - ... - \frac{\delta}{2\pi} d\theta_{n-1} $. Hence for the Lagrangian $ \hat{L} = \iota'(\tilde{L}) \subset W(\epsilon'/4) \times T^* \mathbb{T}^{n-1} \subset T^* \mathbb{T}^n $ and the loops $ \hat{\beta_j} := \iota' \circ \tilde{\beta}_j $, $ j = 1,...,n $, generating $ \pi_1(\hat{L}) $, we have $ \lambda'(\hat{\beta}_1) = \omega(u) $, and $ \lambda'(\hat{\beta}_j) = 0 $, $ 2 \leqslant j \leqslant n $. Now consider the symplectic shift $ \Phi : T^* \mathbb{T}^{n} \rightarrow T^* \mathbb{T}^{n} $ given by $ \Phi(q_0,p_0,q_1,p_1,...,q_{n-1},p_{n-1}) = (q_0,p_0 - \omega(u),q_1,p_1,...,q_{n-1},p_{n-1}) $. Then it follows that the shifted Lagrangian $ \Phi(\hat{L}) \subset T^* \mathbb{T}^n $ is exact, and moreover it does not intersect the zero-section, since $ \hat{L} \subset W(\epsilon'/4) \times T^* \mathbb{T}^{n-1} $ and $ | \omega(u) | \geqslant \epsilon' > \epsilon' / 4 $. However, by the theorem of Gromov~\cite{G} (see section $2.3.\text{B}_4''$ in~\cite{G}), a closed exact Lagrangian 
submanifold of a cotangent bundle must intersect the zero-section. Contradiction.
\end{proof}

Now, Theorem~\ref{T:Main-result} is a straightforward consequence of Proposition~\ref{P:main-step-proof-main-result}:

\begin{proof}[Proof of Theorem~\ref{T:Main-result}]

By Proposition~\ref{P:main-step-proof-main-result}, for any $ \tilde{\phi} \in \widetilde{H}_0 $ we have $ \Flux (\tilde{\phi}) \in 
\overline{\Gamma^{a}_{top}} = \Gamma^{a}_{top}  + \overline{\Gamma^{e}_{top}} \subseteq \Gamma_{top}  + \overline{\Gamma^{e}_{top}} $, where $ a = ev'(\tilde{\phi}) \in \pi_1(M) $. Hence $ \Gamma_0 \subseteq \Gamma_{top}  + \overline{\Gamma^{e}_{top}} $.

\end{proof}

Now we turn to the proof of Theorem~\ref{T:flux-lipschitz}.

\begin{proof}[Proof of Theorem~\ref{T:flux-lipschitz}]
It is clearly enough to prove that for any smooth embedded loop $ \gamma : [0,1] \rightarrow M $, there exist constants $ c = c(M,\omega,g,\gamma), C = C(M,\omega,g,\gamma) $, such that for any path $ \phi^t $, $ t \in [0,1] $ of symplectomorphisms of $ M $, $ \phi^0 = \id_M $, $ \phi^1 = \phi $, with $ \max_{t \in [0,1]} d(\id_M,\phi^t) < c $, we have $ | \Flux(\{ \phi^t \})(\gamma) | \leqslant C d(\id_M,\phi) $. Now fix a smooth embedded loop $ \gamma : [0,1] \rightarrow M $. Then a neighbourhood of $ \gamma([0,1]) $ in $ M $ is standard, and hence for some $ \epsilon > 0 $ one can find a symplectic embedding $ \iota : W(\epsilon) \times B(\epsilon)^{\times n-1} \rightarrow M $, such that $ \iota(S^1 \times \{ 0 \} \times ... \times \{ 0 \}) = \gamma([0,1]) $. We set $$ c_1 = d( \iota( W(\epsilon / 2) \times S(\epsilon /3)^{\times n-1} ),  M \setminus \iota (W(\epsilon) \times  B'(\epsilon)^{\times n-1}  ) ) $$ (recall that the notation $ B'(\epsilon) = B(\epsilon) \setminus \{ 0 \} \subset \mathbb{R}^2 $ stands for the punctured disc). Now let $ \phi^t $, $ t \in [0,1] $ be a path of symplectomorphisms of $ M $, $ \phi^0 = \id_M $, $ \phi^1 = \phi $, with $ \max_{t \in [0,1]} d(\id_M,\phi^t) < c_1 $. Define the ``flux function" $ \kappa : [0,1] \rightarrow \mathbb{R} $ by $ \kappa (t) = \Flux(\{ \phi^{s} \}_{s \in [0,t]})(\gamma) $. We assume that $ \kappa(1) = \Flux(\{ \phi^t \})(\gamma) \neq 0 $. If $ \max_{t \in [0,1]} |\kappa(t)| > \epsilon / 3 $, then we define $ T \in [0,1] $ to be minimal such that $ | \kappa(T) | = \epsilon / 3 $, otherwise we set $ T = 1 $. We have $ | \kappa (t) | \leqslant \epsilon /3 $ for all $ t \in [0,T] $. Now, on $ \iota ( W(\epsilon) \times B(\epsilon)^{\times n-1} ) $, consider the time dependent vector field $ X_\kappa^t = \kappa'(t) \frac{\partial}{\partial p_0} $, where $ (q_0,p_0,x_1,y_1,...,x_{n-1},y_{n-1}) $ are the standard coordinates on $ W(\epsilon) \times B(\epsilon)^{\times n-1} $. Denote by $ X^t $ the time dependent symplectic vector field of the flow $ \phi^t $. The difference $ Y^t = X^t - X_\kappa^t $, $ t \in [0,T] $, is a time dependent Hamiltonian vector field on $ \iota( W(\epsilon) \times B(\epsilon)^{\times n-1} ). $ With help of a cut-off, we can find a time dependent Hamiltonian vector field $ Z^t $, $ t \in [0,T] $ on $ M $, such that $ Y^t(x) = Z^t(x) $ for any $ x \in \iota ( W(\epsilon/2) \times B(\epsilon/2)^{\times n-1} ) $ and $ t \in [0,T] $. Now look at the time dependent symplectic vector field $ \tilde{X}^t = X^t - Z^t $, $ t \in [0,T] $, on $ M $, and denote by $ \tilde{\phi}^t $, $ t \in [0,T] $ its symplectic flow on $ M $. Then $ \psi^t := (\phi^t)^{-1} \circ \tilde{\phi}^t $, $ t \in [0,T] $ is a Hamiltonian flow on $ M $ since it has zero flux at all times. Consider the Lagrangian $ L = \iota( S^1 \times S(\epsilon /3) \times ... \times S(\epsilon /3) ) \subset \iota (W(\epsilon) \times B(\epsilon)^{\times n-1} ) \subset M $. We have that $ \tilde{X}^t (x) = X_\kappa^t (x) $ for any $ x \in \iota ( W(\epsilon/2) \times B(\epsilon/2)^{\times n-1} ) $ and $ t \in [0,T] $, and hence $ \tilde{\phi}^t(\iota(q_0,p_0,x_1,y_1,...,x_{n-1},y_{n-1})) = \iota(q_0,p_0+ \kappa(t),x_1,y_1,...,x_{n-1},y_{n-1}) $ whenever $ \iota(q_0,p_0,x_1,y_1,...,x_{n-1},y_{n-1}) \in L $ and $ t \in [0,T] $, and so for any $ t \in [0,T] $, $ \tilde{\phi}^t (L) $ is obtained from $ L $ by shifting it by $ \kappa(t) $ in the ``$ p_0 $ direction". Clearly there exists a constant $ c_2 = c_2(M,\omega,g,\gamma,\iota) $, such that the distance $ d(\tilde{\phi}^{T}(L),L) $ is greater than or equal to $ c_2 | \kappa(T) | $. Also note that since $ d(\id_M,\phi^t) < c_1 $ for any $ t $, we get that $ \psi^t (L) = (\phi^t)^{-1} \circ \tilde{\phi}^t (L) \subset \iota( W(\epsilon) \times B'(\epsilon)^{\times n-1} ) $ for any $ t \in [0,T] $. Now, $ L $ cannot be Hamiltonianly displaced {\em inside} $ \iota( W(\epsilon) \times B'(\epsilon)^{\times n-1} ) $, and so we must have $ \psi^T (L) \cap L = (\phi^T)^{-1} \circ \tilde{\phi}^T (L) \cap L \neq \emptyset $. Since in addition we have $ d(\tilde{\phi}^T(L),L) \geqslant c_2 | \kappa (T) | $, we conclude that $ d(\id_M,\phi^T) \geqslant c_2 | \kappa(T) | $.

We have shown that for any path $ \phi^t $, $ t \in [0,1] $ of symplectomorphisms of $ M $, $ \phi^0 = \id_M $, $ \phi^1 = \phi $, with $ \max_{t \in [0,1]} d(\id_M,\phi^t) < c_1 $, we must have $ d(\id_M,\phi^T) \geqslant c_2 |\kappa(T)| $. Thus, if we set $ c = \min(c_1, c_2 \epsilon /3) $ and $ C = 1/ c_2 $, then for any path $ \phi^t $, $ t \in [0,1] $ of symplectomorphisms of $ M $, $ \phi^0 = \id_M $, $ \phi^1 = \phi $, with $ \max_{t \in [0,1]} d(\id_M,\phi^t) < c $, we have $ c_2 \epsilon /3 \geqslant c > d(\id_M,\phi^T) \geqslant c_2 |\kappa(T)| $, hence $ |\kappa(T)| \neq \epsilon /3 $, which means that $ T = 1 $, and we therefore get $ | \Flux(\{ \phi^t \})(\gamma) | = | \kappa (1) | = | \kappa(T) | \leqslant C d(\id_M,\phi^T) = C d(\id_M,\phi) $.

\end{proof}


\bigskip
\noindent Lev Buhovski\\
School of Mathematical Sciences, Tel Aviv University \\
{\it e-mail}: levbuh@post.tau.ac.il
\bigskip

\end{document}